\documentclass{bcp}
\usepackage{amsmath,amsthm}
\usepackage{amssymb}

\usepackage[T1]{fontenc}
\usepackage[english]{babel}

\usepackage{enumitem}

\newtheorem{theorem}{Theorem}[section]

\newtheorem{lemma}[theorem]{Lemma}

\newtheorem{conjecture}[theorem]{Conjecture}

\theoremstyle{definition}

\newtheorem{example}[theorem]{Example}
\newtheorem{remark}[theorem]{Remark}

\newcommand{\R}{\mathbb{R}}

\newcommand{\Z}{\mathbb{Z}}
\newcommand{\N}{\mathbb{N}}

\newcommand{\ball}[2]{B\left(#1,#2\right)}

\newcommand{\diam}{\operatorname{diam}}

\renewcommand{\d}[1]{d\left(#1\right)}

\newcommand{\Frechet}{Fr\'{e}chet }

\long\def\pag{\beg=\count0
\def\phead{}}

\long\def\maketitlebcp{\def\and{{\lowercase{and}}}\pag
\vfootnote{}{\noindent\kern-15pt 2020 {\it Mathematics Subject
Classification}\/: \mclass}%
\vfootnote{}{\noindent\kern-15pt {\it Key words
and phrases}\/: \kwords}
\ifnum\tnum=100 \unhbox\thanksbox\fi
\vfootnote{}{}{}
\vglue7cc
{\large\bf\centering\settitle\par
}\box\abbox
\ded
\vskip8pt}

\begin{document}
\mathclass{Primary 54E25, 54E35.}

\title{How much can we extend the Assouad embedding theorem?}
\abbrevauthors{F. Turoboś and O. Dovgoshey}
\abbrevtitle{On extending Assouad embedding theorem}

\author{Filip Turobo\'s}
\address{Institute of Mathematics,\\
Lodz University of Technology\\
Łódź, Poland\\
e-mail: filip.turobos@p.lodz.pl}
\author{Oleksiy Dovgoshey}
\address{Institute of Applied Mathematics and Mechanics of NASU\\
Slovyansk, Ukraine and\\
Institut fuer Mathematik Universitaet zu Luebeck\\
Lubeck, Germany\\
e-mail: oleksiy.dovgoshey@gmail.com}

\maketitlebcp           
\begin{abstract}
The celebrated Assouad embedding theorem has been known for over 40 years. It states that for any doubling metric space (with doubling constant $C_0$) there exists an integer $N$, such that for any $\alpha\in (0.5,1)$ there exists a positive constant $C(\alpha,C_0)$ and an injective function $F:X\to \mathbb{R}^N$ such that
\[
\forall x,y\in X \quad C^{-1} d(x,y)^{\alpha} \leq \|F(x)-F(y) \|  \leq Cd(x,y)^{\alpha}
\]

In the paper we use the remetrization techniques to extend the said theorem to a broad subclass of semimetric spaces. We also present the limitations of this extension -- in particular, we prove that in any semimetric space which satisfies the claim of the Assouad theorem, the relaxed triangle condition holds as well, which means that there exists $K\geq 1$ such that
\[
\forall x,y,z\in X \quad d(x,z)\leq K\left( d(x,y)+d(y,z)\right).
\]
\end{abstract}
\section{Introduction}

The early concept of the triangle inequality can be traced back to Euclid's Elements \cite{Euclid1956}, which date back to 300 BC. Regardless, it took the mathematicians over 20 centuries to formalize the notion of a metric space \Frechet \cite{Frechet1906}. The subsequent results which further explored and generalized the concept of distance appeared (among others) in the works of Chittenden \cite{Chittenden1917} and Wilson \cite{Wilson1931}. It should also be noted here that the semimetric spaces were first considered by \Frechet in \cite{Frechet1906}, where he referred to them by the name ``\textit{classes E}''. The notion of semimetric space arises if we drop the triangle inequality from the formal definition of distance. They are very interesting object on its own -- their inherent, chaotic nature is alluring, considering how closely they seem to be related to their elegant metric counterparts.

Due to the erratic behavior of such spaces, additional, weaker versions of the triangle inequality started to emerge. These can be again traced back to Chittenden \cite{Chittenden1917}, \cite{Chittenden1927} and Wilson \cite{Wilson1931}, but are continuously developed even nowadays. In order to bring down the semimetric space analysis to a well-known metric ground, the mathematicians started to explore the possibilities of remetrization of such spaces or embedding them in a well-known metric space, after slight alteration of the initial distance concept. Regarding the first possibility, in 1937, pursuing the works of Chittenden \cite{Chittenden1927}, Frink \cite{Frink1937} provides the existence of an Lipschitz-equivalent metric to a given quasimetric\footnote{To the best of our knowledge, the concept of such relaxation was formally introduced in the works of Bourbaki \cite{Bourbaki1966}, however the name ``$b$-metric space'' is often associated with Czerwik \cite{Czerwik1993}.} (with relaxation constant $K = 2$). This result is then explored in the papers of Schroeder \cite{Schroeder2006} as well as thoroughly improved in the paper of Chrz\k{a}szcz et al. \cite{Chrzaszcz2018}. Similar results, albeit in a slightly different spirit, can be found in the (earlier) paper of Fagin et al. \cite{Fagin2003}.

The second of the discussed paths was paved by Assouad \cite{Assouad1979,David2013}, who proved that for a doubling metric space\footnote{That is, a space, where every ball of any positive radius $r$ can be covered by a fixed, finite number of balls with radius $\frac{r}{2}$. Formal definition appears in the subsequent section.}   it is possible to embed it into an euclidian space of some sufficiently high dimension, after altering the initial distance by raising it to some power $0<\alpha < 1$. This celebrated result bloomed into a whole theory of doubling spaces and allowed to further explore their relationship with the so-called doubling measures (more on that topic can be found, for example, in the nice lectures of Heinonen, see \cite[Chapters 1,10,13]{Heinonen2001}). In this short paper we explore the possibilities of extending this theorem to a broader scope of semimetric spaces via the use of Frinks metrization technique. We also prove that extending it beyond certain point is meaningless, as the statement of the theorem enables us to draw conclusions on the type of space we are dealing with.

The paper is organized as follows: in the subsequent section we introduce the basic definitions and known results from past papers. We then proceed with the first main part of the article, which focuses on extending the Assouad theorem. Fourth section deals with the limitations of such extensions, while the last part of the paper poses some open problems and draws the directions of further research in this area.

\section{Preliminaries}

Throughout the paper we denote by $\R$ and $\N$ the set of real numbers and positive integers, respectively. We begin with the definition of a semimetric space.

Let $X\neq \emptyset$ and $d:X\times X \to [0,+\infty)$ be a function satisfying the following conditions for all $x,y\in X$:
\begin{itemize}
\item[(S1)] $\d{x,y}=0 \iff x=y$;
\item[(S2)] $\d{x,y}=\d{y,x}$.
\end{itemize}
Then the pair $(X,d)$ is called a \textbf{semimetric space} (and $d$ is called a \textbf{semimetric}). If additionally there exists $K\geqslant 1$ such that $d$ satisfies
\begin{itemize}
\item[(B)] $\d{x,z}\leqslant K\left( \d{x,y} + \d{y,z}\right)$ for all $x,y,z \in X$
\end{itemize}
then $(X,d)$ is said to be a \textbf{$b$-metric space} with \textbf{relaxation constant} $K$. If one can pick $K=1$ we obtain the definition of \textbf{metric space}.

In semimetric scope the classic notions of sequence convergence, open/closed ball, completeness and so forth are identical to their counterparts in metric spaces. Nevertheless, they tend to have slightly less regular properties. More on that topic can be found e.g. in \cite{Chrzaszcz2018prim}. We know quote the two refined versions of Frink's theorem \cite{Chrzaszcz2018}. The first one differs from the versions from \cite{Frink1937,Schroeder2006} by having tightened bounds on the original distance $d$.

\begin{theorem}[\textbf{Frink's theorem (revisited)}]\label{thm:specific Frink bounds}
Let $(X,d)$ be a $b$-metric with relaxation constant $K\leqslant 2$. Then, there exists a metric $D$ for which
\[
\forall x,y\in X \quad D(x,y) \leqslant d (x,y) \leqslant K^2 D(x,y).
\]
\end{theorem}

The second version allows us to drop the requirement on $K$ by introducing H\"older-type relationship between the original distance and its metric counterpart.

\begin{theorem}[\textbf{Frink's remetrization technique for $b$-metric spaces}]\label{thm:generalFrink}
Let $(X,d)$ be a $b$-metric space with relaxation constant $K\geqslant 1$. Then for any $\varepsilon > 0$, there exist $p\in (0,1]$ and a metric $D$ for which
\[
\forall x,y\in X \quad D(x,y) \leqslant d^p (x,y) \leqslant (1+\varepsilon) D(x,y).
\]
\end{theorem}

The notion of doubling property (which is supposed to resemble the finite-dimensiona\-lity for normed spaces) in the case of semimetric spaces is not well-explored in general. We first introduce two notions of doubling property (instead of one, as it is in metric case) and in the subsequent part we explain where did this difference came from. Let us begin with the notion of \textbf{diameter}, which for a subset $A\subset X$ of a semimetric space $(X,d)$ is defined as $\diam A :=\sup_{x,y\in A} d(x,y)$. About a semimetric space $(X,d)$ one can say that it is:
\begin{itemize}
\item \textbf{weakly doubling} if there exists a constant $C\in \N$, such that any bounded set $D$ with diameter $l<\infty$ can be covered by at most $C$ sets with diameters bounded by $\frac{l}{2}$.
\item \textbf{doubling} if there exists a constant $C\in \N$, such that for any $x\in X$ and positive radius $r$, an open ball $\ball{x}{r}$ can be covered by open balls $\ball{x_i}{\frac{r}{2}}$, where $x_i\in X$ for $i\in \{1,\dots, C\}$.
\end{itemize}

These notions clearly coincide in the scope of metric spaces (since every bounded set can be put inside a ball with radius equal to its diameter; moreover, every ball is a bounded set). In the semimetric scope this is not the case, hence the distinction. An appropriate example at the beginning of the subsequent section. We conclude this part by quoting the celebrated Assouad theorem we are going to investigate.

\begin{theorem}[\textbf{Assouad embedding Theorem}]\label{thm:Assouad Theorem}
Let $(X,d)$ be a doubling metric space with doubling constant $C_0\geq 1 $.
There exists an integer $N\in \mathbb{N}$ such that for $\alpha \in \left(\frac{1}{2},1\right)$ there is a constant $C$ (which depends only on $C_0$ and $\alpha$) and an injective function $F\colon X\to \R^N$, for which the following inequalities hold:
\begin{equation}\label{eq:Assouad_Bound}
\forall x,y\in X \quad C^{-1}d^\alpha(x,y) \leqslant \| F(x) - F(y) \| \leqslant C d^\alpha(x,y)
\end{equation}
\end{theorem}

\section{Doubling semimetric spaces and the Assouad theorem}

We begin this part of our paper with the promised example, which presents a semimetric space which is weakly doubling but fails to be doubling.

\begin{example}\label{ex:31}
Let $\mathbb{Z}$ be the set of all integer numbers and let $(X,d)$ be a semimetric space with $X=\mathbb{Z}$ and such that
\begin{equation}\label{eq:nondoubling semimetric on Z}	
d(x,y) =
\begin{cases}
1, & \text{if } x = 0 \text{ and } y \in \mathbb{Z} \setminus \{0\}\\
|x-y|, & \text{if } x,y\in \mathbb{Z}\setminus \{0\} \\
0, & \text{if } x=y=0.
\end{cases}
\end{equation}

Consider an open ball $\ball{0}{r}$ with radius $r\in (1,2)$. Since $d(0,z) = 1 < r$ for any non-zero integer $z$, then $B(0,r) = \mathbb{Z}$. Consider the open ball with radius $r^\prime = \frac r 2$ and center $x\in X$. Since $r^\prime \in (\frac 1 2, 1)$, then $B(x,r^\prime)$ cannot contain any other element than $x$ itself (every two elements of $\Z$ are at least $1$ unit of distance apart). Consequently, one cannot cover $B(0,r)$ with finite family of balls with the radii at most $r^\prime$, thus space fails to be doubling.

Now, we can prove that the space \((X, d)\) is weakly doubling. Obviously, every bounded set in this space is finite, due to $X$ being essentially an isometric copy of $\mathbb{Z}\setminus\{0\}$ equipped with an Euclidean distance with one, additional point. Therefore, the fact that we can cover any set of finite diameter with just several sets of lower diameter is rather clear. To stay within the formal setting, we claim that $C \leqslant 3$ which proves its finiteness.

Let $A\subset X$ be an arbitrary set of finite diameter. If $A$ contains at most $3$ elements, the choice of covering sets is obvious (since we can use singletons, which have zero diameter). In other case, $A$ contains at least two distinct numbers from $\mathbb{Z}\setminus\{0\}$.

Therefore, let $m,M$ be minimal and maximal (respectively) elements of $A\cap\left(\Z\setminus\{0\} \right)$. Since $d(0, z)=1$ for $z\in \Z\setminus\{0\}$, then one can easily notice that $\diam A = M - m$. Consider sets
\begin{itemize}
\item $A_1:=\{0\};$
\item $A_2:=\overline{B}\left(M,\frac{M-m}{2} \right) \cap \{m, m+1, \dots, M \};$
\item $A_3:=\overline{B}\left(m,\frac{M-m}{2} \right) \cap \{m, m+1, \dots, M \},$
\end{itemize}
where $\overline{B}\left(M,\frac{M-m}{2} \right)$ and $\overline{B}\left(m,\frac{M-m}{2} \right)$ are closed balls with radii $\frac{M-m}{2}$ and centers $M$ and $m$ respectively.

One can see that $A=A_1\cup A_2\cup A_3$, since the first set takes care of possible point outside of $\mathbb{Z}\setminus\{0\}$ and latter ones contain the integer part of our set. As a result, $C\leqslant 3$.

\end{example}

Clearly, an assumption which ties the notion of boundedness with the possibility of containing a bounded set within a ball with sufficiently large radius is enough to show the equivalence of weak doubling and doubling properties. The triangle-like condition assumed in the definition of $b$-metric space is an example of such premise. The following questions naturally arises as a follow-up:
\textit{Does doubling property imply that the semimetric space is weakly doubling  as well? If not, what additional assumptions need to be made?}

A rather easy observation might yield some answers. It can be stated as follows:

\textbf{Observation:}
Let $(X,d)$ be a semimetric space. For every nonempty, bounded set $A$, there exists an open ball of finite radius $r>0$ and center $x\in A$ such that $B(x,r) \supset A$ (one can pick $r:=\diam (A) +1$ for instance).

We have managed to find an interesting example of a space which enjoys the doubling property while fails to be weakly doubling.

\begin{example}
Let $(X,d)$ be a metric space which is not doubling and let $X\cap \mathbb{N} = \emptyset$. Let us define a semimetric $\rho$ on $Y=X\cup\mathbb{N}$ such that \[
\rho(x,y) = \begin{cases}
d(x,y), & \text{if } x, y\in X\\
\max\{\frac{1}{x}, \frac{1}{y}\}, & \text{if } x, y\in \mathbb{N} \text{ and } x \neq y\\
\frac{1}{x}, & \text{if } x \in \mathbb{N} \text{ and } y \in X.
\end{cases}
\]
Then $(Y,\rho)$ is doubling with doubling constant $C=1$ (since every ball with a center in $X$ contains some natural numbers as well, and thus can be easily covered by a ball centered at sufficiently large $n\in\mathbb{N}$) but it is not weakly doubling because any subspace of a weakly doubling semimetric space is also weakly doubling, while $X$ clearly fails to satisfy this condition.
\end{example}

Considering some Banach space of infinite dimension with metric induced by norm in the role of $X$ yields a slightly more tangible example, if one deems it necessary.

Now we will show that the scope of Theorem \ref{thm:Assouad Theorem} can be extended to the class of $b$-metric spaces. This will be achieved via the use of remetrization techniques provided by Frink theorem. What we need to establish is whether they preserves the doubling property (perhaps with a different constant). This leads us to the following pair of lemmas:

\begin{lemma}[\textbf{On metric transformation which preserves the doubling condition}]\label{lemma:doubling lemma}
Let $(X,d)$ be a semimetric space with doubling constant $C$ and let $p>0$. Then $(X,d^p)$ is also a doubling semimetric space with a doubling constant $C'=C^n$, where $n=\lceil r \rceil$.
\end{lemma}
\begin{proof}
We start with pointing out that the obvious equivalence
\begin{equation}\label{eq:equivalence of d and dp}
d(x,y)<r \iff d^p(x,y) < r^p
\end{equation}
for any $x,y\in X$ and $r>0$. Let $n\in \N$ be large enough to $2^{np} \geqslant 2$ (taking $n=\lceil \frac{1}{p} \rceil$ is sufficient for our task). Repeated usage of the doubling property yields another rather obvious property of a doubling space, which is:
\[
B_d(x,r) \subset \bigcup_{i=1}^{C^n} B_d(x_i, \frac{r}{2^n})
\]
for some $(x_i)_{i=1}^{C^n}$, $x_1, \dots x_{C^n} \in X$. Now, to prove that $(X,d^p)$ is doubling, fix an arbitrary $s>0$ and put $r:=s^{\frac{1}{p}}$.
\[
B_{d^p}(x,s) = B_d(x,r) \subset \bigcup_{i=1}^{C^n} B_d(x_i, \frac{r}{2^n}) =
\bigcup_{i=1}^{C^n} B_{d^p}(x_i, \frac{s}{2^{np}}) \subset
\bigcup_{i=1}^{C^n} B_{d^p}(x_i, \frac{s}{2}).
\]
Since $s$ was taken arbitrarily, then $C':=C^n$ is a doubling constant for $(X,d^p)$.
\end{proof}

The second lemma will enable us to use Theorem \ref{thm:generalFrink} as our remetrization technique which preserves both the topological structure and the doubling property.

\begin{lemma}[\textbf{On preserving doubling property by bounding semimetric}]\label{lemma:doubling Frink approach}
Let $(X,d)$ be a doubling semimetric space. If a semimetric $D$ is bi-Lipschitz equivalent to $d$, that is there exists a positive constant $\alpha\geqslant 1$ such that
\begin{equation}\label{eq:boundTheoremDoubling}
\forall x,y\in X \quad	D(x,y) \leqslant d(x,y) \leqslant \alpha D(x,y)
\end{equation}
holds, then $(X,D)$ is a doubling semimetric space as well.
\end{lemma}
\begin{proof}
Let $B_{D}(x,r)$ be an arbitrary open ball in the semimetric space $(X,D)$. Due to the right-hand side of the inequality of \eqref{eq:boundTheoremDoubling}, we see that
\[
D(z,y) < s \implies d(z,y)<\alpha s
\]
for any two points $y,z\in X$ and a positive real $s$. Thus, the following inclusion is justified:
\[
B_{D}(x,r)\subset B_{d} (x,\alpha r).
\]
Now we apply the doubling property $N$-fold (where $N\in\N$ satisfies $\alpha<2^{N-1}$) to obtain a covering of the latter ball by a finite amount of open balls $B_{d}\left(x_i, \frac{r}{2}\right)$, $i = 1,2,\dots, C_0^N$, where $C_0$ is a doubling constant of $d$.

Now, the first part of \eqref{eq:boundTheoremDoubling} yields us the following inclusion for any $z\in X$ and $s>0$:
\[
B_{d}(z,s) \subset B_D(z,s).
\]
Summing it up, we obtain
\[
B_D(x,r) \subset B_{d}(x,\alpha r)\subset \bigcup_{i=1}^{C_0^N} B_{d}\left(x_i, \frac{r}{2}\right) \subset \bigcup_{i=1}^{C_0^N} B_{D}\left(x_i, \frac{r}{2}\right).
\]
Consequently, $(X,D)$ has the doubling property.
\end{proof}

Having all the above lemmas, we are ready to present an extended version of Assouad theorem.

\begin{theorem}[\textbf{Assouad theorem for $b$-metric spaces}]
Let $(X,d)$ be a $b$-metric space with doubling property. There exists an integer $N\in\mathbb{N}$ such that for $\alpha \in \left( 0, 1 \right)$, there are a constant $C$ and an injective function $F:X\to \R^N$, for which the following inequalities hold for any $x,y\in X$:
\begin{equation}\label{eq:general bounds Assouad}
C^{-1} d^\alpha(x,y) \leqslant \| F(x) - F(y) \| \leqslant C d^\alpha(x,y).
\end{equation}
\end{theorem}

\begin{proof}
From Theorem \ref{thm:generalFrink} for $\varepsilon=1$ we obtain the existence of a metric $D$ and a power $p\in (0,1]$ such that for all $x,y\in X$
\begin{equation}\label{eqn:bounds_in_the_Assouad_Theorem}
D(x,y)\leqslant d^p(x,y) \leqslant 2 D(x,y).
\end{equation}
From Lemma \ref{lemma:doubling lemma} we obtain that $d^p$ is a doubling semimetric (since $d$ was doubling). In turn, from Lemma \ref{lemma:doubling Frink approach} we infer that doubling property of $d^p$ guarantees the same property for $D$. Applying Theorem \ref{thm:Assouad Theorem} yields the existence of real constants $C>0$, $\alpha\in\left(\frac{1}{2},1\right)$ and an embedding $F:X\to \mathbb{R}^N$ for which
\[
C^{-1}\cdot D^\alpha(x,y) \leqslant \|F(x)-F(y)\| \leqslant C\cdot D^\alpha(x,y)
\]

Notice that bounds from \eqref{eqn:bounds_in_the_Assouad_Theorem} yield
\[
C^{-1}\cdot D^\alpha(x,y) \geqslant C^{-1}\cdot \left(\frac{d(x,y)^p}{2}\right)^{\alpha} = \left(2^\alpha C\right)^{-1} \cdot d^{p\alpha}(x,y),
\]
as well as
\[
C\cdot D^\alpha(x,y) \leqslant C\cdot d(x,y)^{p\alpha} \leqslant \left(2^{\alpha} C \right) \cdot d^{p\alpha}(x,y).
\]
From these two inequalities we conclude that desired inequalities \eqref{eq:general bounds Assouad} hold if we substitute the constants for $\alpha':=p\alpha$ and $C':=2^\alpha \cdot C$.
\end{proof}

Thus we have extended the range of Assouad theorem to the scope of $b$-metric spaces which have the doubling property.

\section{The limitations}

We might wonder, whether it is possible for some broader class of doubling semimetric spaces to be bi-Lipschitz embeddable in $\R^N$ after raising the discussed distance function to a power. It turns out that being a $b$-metric space is a necessary condition for a space to be bi-Lipschitz embeddable in any metric space, not just $\R^N$. This claim is put in a formal statement in the following

\begin{theorem}[\textbf{On embedding semimetric spaces in metric spaces}]\label{thm:limits of Assouad Theorem}
Let $(X,d)$ be a semimetric space and assume that for some power $\alpha\in (0,1]$, there exists a mapping $F:X\to Y$ which is a bi-Lipschitz embedding of $(X,d^\alpha)$ into a metric space $(Y,d_Y)$. Then $(X,d)$ is a $b$-metric space.
\end{theorem}
\begin{proof}
Let $C$ be the bi-Lipschitz constant of the embedding $F$, i.e.
\[
C^{-1}\cdot d^\alpha(x,y) \leqslant d_Y\left(F(x),F(y)\right) \leqslant C\cdot d^\alpha(x,y)
\] 	
Let $x,y,z \in X$  be an arbitrary triplet of points. We have
\begin{eqnarray*}
d^\alpha(x,z) &\leqslant& C\cdot d_Y\left(F(x),F(z)\right)\\
& \leqslant&
C \cdot \left( d_Y\left(F(x),F(y) \right) + d_Y\left(F(y),F(z)\right) \right)\\
&\leqslant& C^2\cdot \left( d^\alpha(x,y) + d^\alpha(y,z) \right)\\
&\leqslant& 2 C^2 \cdot \max\left\{d^\alpha(x,y), d^\alpha(y,z) \right\},
\end{eqnarray*}
where the last inequality comes from the fact, that a sum of two numbers is bound from above by twice their maximum.
Therefore we arrive at
\[d^\alpha(x,z) \leqslant 2 C^2 \cdot \max\left\{d^\alpha(x,y), d^\alpha(y,z) \right\}.
\]
If we raise both sides to the power $\frac{1}{\alpha}$, we will obtain the desired conclusion
\begin{equation*}
d(x,z) \leqslant 2^\frac{1}{\alpha} C^\frac{2}{\alpha} \cdot \max\left\{d(x,y), d(y,z) \right\}\leqslant 2^\frac{1}{\alpha} C^\frac{2}{\alpha} \cdot \left( d(x,y) + d(y,z) \right). \qedhere
\end{equation*}
\end{proof}

\begin{remark}\label{r4.2}
The procedure of raising a metric to the power $\alpha\in (0,1)$ is sometimes referred to as snowflaking (see \cite{David2013, TWa}) and this is a special case of transforming of metrics by ultrametric-preserving functions (see \cite{BiletDovgosheyShanin2021, Dovgoshey2020, PongsriiamTermwuttipong2014, VallinDovgoshey2021}).
\end{remark}

Theorem~\ref{thm:limits of Assouad Theorem} strongly resembles the following theorem, which is due to Fagin et al. \cite[Theorem 4.4]{Fagin2003}:

\begin{theorem}[\textbf{Fagin et al. metric boundedness condition}]
A semimetric space $(X,d)$ satisfies the $c$-relaxed polygonal inequality for some $c\geqslant 1$, i.e.,
\[
\forall n\in\N \quad \forall x_1,\dots,x_n\in X \quad d(x_1,x_n)\leqslant c \sum_{i=1}^{n-1} d(x_i,x_{i+1});
\]
if and only if there exists a metric $D$ on $X$ such that
\[
\forall x,y \in X \quad D(x,y) \leqslant d(x,y)\leqslant c D(x,y).
\]

\end{theorem}

\section{Concluding remarks and open problems}

The measure-theoretic aspects of doubling and weakly doubling semimetric spaces remain (to the best of our knowledge) undiscovered. Perhaps additional research in this area might yield interesting results, which will give us a better insight in the nature of such objects. Nevertheless, little is known about generalizations of Assouad theorem considering the remaining two axioms of metric. This might also provide intriguing research opportunities.

Regarding the difference between weakly doubling and doubling properties, after revisiting the remarks below Example \ref{ex:31}, one can ask a following question:

\textbf{Open problem:} What are the weakest assumptions under which the notions of weakly doubling and doubling semimetric space coincide?

\begin{conjecture}[prove or disprove]
Let \((X, d)\) be a semimetric space. Then the following statements are equivalent:
\begin{enumerate}
\item [\((i)\)] Every ball is a bounded subset of \((X, d)\).
\item [\((ii)\)] The space \((X, d)\) is doubling if and only if it is weakly doubling.
\end{enumerate}
\end{conjecture}

\section{Additional notes}

Oleksiy Dovgoshey was partially supported by Volkswagen Stif\-tung Project ``From Modeling and Analysis to Approximation''.

Filip Turoboś has been supported by the Polish National Agency for Academic Exchange under the NAWA Urgency Grant programme.

\end{document}